\documentclass[11pt]{amsart}
\usepackage{latexsym}
\usepackage{amsfonts}

\newcommand{\field}[1]{\mathbb{#1}}

\newcommand{\K}{\field{K}}

\newtheorem{defi}{Definition}

\newtheorem{lem}[defi]{Lemma}
\newtheorem{theo}[defi]{Theorem}
\newtheorem{co}[defi]{Corollary}

\newtheorem{re}[defi]{Remark}

\font\tenmsy=msbm10

\def\Bbb#1{\hbox{\tenmsy#1}} 
\title[Elimination ideals and B\'ezout relations]{Elimination ideals and B\'ezout relations} \makeatletter

\@addtoreset{equation}{section}
\makeatother
\author{Andre Galligo \& Zbigniew Jelonek}
\address[A. Galligo]{Universit\'e C\^ote d'Azur, LJAD, INRIA\\
France}
\email{galligo@unice.fr}

\address[Z. Jelonek]{Instytut Matematyczny PAN\\
\'Sniadeckich 8, 00-656 Warszawa\\
Poland}
\email{najelone@cyf-kr.edu.pl}

\keywords{Nullstellensatz, polynomials,  elimination, affine variety}

\subjclass{14 D 06, 14 Q 20.}

\date{}

\begin{document}

\maketitle

\begin{abstract}
Let $k$ be an infinite field and $I\subset k [x_1, \ldots ,x_n]$ be a non-zero ideal such that dim $V(I)=q\ge 0$.
Denote by  $(f_1, \ldots, f_s)$ a set of generators of $I$. 
One can  see that in the set $I\cap k [x_{1},...,x_{q+1}]$ there exist  non-zero polynomials,
depending only on these $q+1$ variables. 
We aim to bound the minimal degree of the polynomials of this type, and of a B\'ezout (i.e. membership) relation  expressing
such a polynomial as a combination of the $f_i$. In particular we show that if $\deg
f_i=d_i$ where $d_1\ge d_2...\ge d_s,$ 
 then there exist  a non-zero
polynomial $\phi (x)\in k[x_{1},...,x_{q+1}]\cap I$, such that $ \deg \phi\le d_s\prod^{n-q-1}_{i=1} d_i.$
\end{abstract}

\bibliographystyle{alpha}

\section{Introduction}
Let $I\subset k [x_1,...,x_n]$ be a non-zero ideal such that dim $V(I)=q\ge 0$. 
Using Hilbert Nullstellensatz we can easily see, that in the elimination ideal $I\cap k [x_1,...,x_{q+1}]$ there exist  non-zero polynomials. 
It is interesting to know the minimal degree of the polynomials in this ideal.  
Here,  performing a generic change of coordinates, and continuing the approach presented in \cite{jel}, 
we get a sharp estimate for the degree of such a minimal polynomial (and also for a corresponding generalized Bezout identity),
in terms of the degrees of generators of the ideal $I$.
Then,  using a deformation arguments we solve the  stated problem. We show that if $\deg
f_i=d_i$ where $d_1\ge d_2...\ge d_s,$ 
 then there exist  polynomials $g_{j}\in k[x_1,\dots,x_n]$  and  a non-zero
polynomial $\phi (x)\in k[x_{1},...,x_{q+1}]$ such that

(a) $ \deg g_{j}f_j\le d_s\prod^{n-q-1}_{i=1} d_i,$

(b) $\phi(x)=\sum^{s}_{j=1} g_{j}f_j.$ 

Note that our result works also in the case dim $V(I)=-1$ (i.e. in the case when $V(I)=\emptyset$) if we put $k[x_0]:=k$ (however our result in this case is a little bit worse than these in  \cite{jel}, \cite{kol}). Hence, from this point of view, we can treat our result as a generalization of the Effective Nullstellensatz.

\section{Main Result}

In this section we present a  geometric construction and establish degree bounds, relying on generic changes of coordinates.
Let us recall (see \cite{jel}) two important tools that we will use in the proof of the main theorem of this section.

\begin{theo}(Perron Theorem)\label{perron}
Let $\Bbb L$ be a field and let $Q_1,\dots,Q_{n+1}\in \Bbb L
[x_1,\dots,x_n]$ be non-constant polynomials with $\deg Q_i=d_i$.
 If the mapping
$Q=(Q_1,\dots,Q_{n+1}) : \Bbb L^n \to \Bbb L^{n+1}$ is generically
finite, then there exists a  non-zero polynomial
$W(T_1,\dots,T_{n+1})\in \Bbb L [T_1,\dots,T_{n+1}]$ such that

(a) $W(Q_1,\dots,Q_{n+1})=0$,

(b) {\rm deg} $W(T_1^{d_1},T_2^{d_2},\dots, T_{n+1}^{d_{n+1}})\le
\prod^{n+1}_{j=1} d_j.$
\end{theo}


\begin{lem}\label{lemat}
Let $\Bbb K$ be an algebraic closed field and let $k\subset \Bbb K$ be its  infinite subfield. Let
$X\subset \Bbb K^m$ be an affine algebraic variety of dimension $n.$
For sufficiently general numbers $a_{ij}\in k$ the mapping
$$\pi: X\ni (x_1,\dots,x_m)\to \Big(\sum_{j=1}^m a_{1j}x_j,
\sum_{j=2}^m a_{2j}x_j,\dots,\sum_{j=n}^m a_{1j}x_j\Big)\in \Bbb
K^n$$ is finite. $\Box$
\end{lem}

In the sequel for a given ideal $I\subset k[x_1,...,x_n]$ by $V(I)$ we mean the set of algebraic zeros of $I$, i.e., 
the zeroes of $I$ in $\Bbb K^n$, where $\Bbb K$ is an algebraic closure of $k.$
Now we can formulate our first main result:

\begin{theo}\label{skzera}
Let $k$ be an infinite field and let $f_1,\dots,f_s\in k[x_1,\dots,x_n]$ be polynomials such  $\deg
f_i=d_i$ where $d_1\ge d_2...\ge d_s.$
Assume that $I=(f_1,\dots,f_s)\in k[x_1,\dots,x_n]$ is a non-zero  ideal, such that $V(I)$ has  dimension $q\ge 0$.   If we take a sufficiently
general system of coordinates $(x_1,\dots,x_n)$, then there exist
polynomials $g_{j}\in k[x_1,\dots,x_n]$  and  a non-zero
polynomial $\phi (x)\in k[x_{1},...,x_{q+1}]$ such that

(a) $ \deg g_{j}f_j\le d_s\prod^{n-q-1}_{i=1} d_i,$

(b) $\phi(x)=\sum^{s}_{j=1} g_{j}f_j.$ 

\end{theo}

\noindent {\it Proof.}  Let $\Bbb K$ be the algebraic closure of $k.$ Take $F_{n-q}=f_s$ and  $F_i=\sum_{j=i}^{s} \alpha_{ij} f_j$ for $i=1,..., n-q-1,$ where $\alpha_{ij}\in k$ are sufficiently general.
Take $J=(F_1,...,F_{n-q}).$ Then $\deg F_{n-q}=d_s$ and deg $F_{i}=d_i$ for $i=1,...,n-q-1.$ Moreover, $V(J)$ has pure dimension $q$ 
and $J\subset I.$ The mapping
$$\Phi : \Bbb K^n\times  \Bbb K\ni (x,z)\to (F_1(x)z,\dots,F_{n-q}(x)z, x)\in  \Bbb K^{n-q}\times  \Bbb K^{n}$$
is a (non-closed)  embedding outside the set
$V(J)\times \K$. Take $\Gamma={\rm cl}(\Phi(\Bbb K^n\times
\Bbb K)).$ Let $\pi : \Gamma\to \Bbb K^{n+1}$ be a generic
projection defined over the field $k.$  Define $\Psi:=\pi\circ\Phi(x,z).$ By Lemma
\ref{lemat} we can assume that
$$\Psi=(\sum^{n-q}_{j=1}
\gamma_{1j} F_jz+l_1(x),\dots, \sum^{n-q}_{j=n-q} \gamma_{n-qj} F_jz+
l_{n-q}(x),l_{n-q+1}(x),..., l_{n+1}(x)),$$ where $l_1,\dots,l_{n+1}$ are  generic
linear form. In particular we can assume that $l_{n-q+i}$, $i=1,..,q+1$ is the
variable $x_i$ in a new generic system of coordinates (of $\Bbb K^n$).

Apply  Theorem \ref{perron} to  $\Bbb L=k(z),$ and to the
polynomials $\Psi_1,\dots,\Psi_{n+1}\in \Bbb L[x]$. Thus there exists a non-zero
polynomial $W(T_1,\dots,T_{n+1})\in \Bbb L [T_1,\dots,T_{n+1}]$
such that
$$W(\Psi_1,\dots,\Psi_{n+1})=0 \ {\rm and}\ \deg
W(T_1^{d_1},T_2^{d_2},\dots,T_k^{d_k},T_{k+1},\dots, T_{n+1})\le
d_s\prod^{n-q-1}_{j=1} d_j,$$ where $k=n-q.$ Since the coefficients of  $W$ are in
$k(z)$, there is a non-zero polynomial $\tilde{W}\in 
k[T_1,\dots,T_{n+1},Y]$ such that

(a) $\tilde{W}(\Psi_1(x,z),\dots,\Psi_{n+1}(x,z),z)=0,$

(b) ${\deg}_T \tilde{W}(T_1^{d_1},T_2^{d_2},\dots,T_k^{d_k},
T_{k+1},\dots, T_{n+1}, Y)\le d_s\prod^{n-q-1}_{j=1} d_j,$ where
${\deg}_T$ denotes the degree with respect to the variables
$T=(T_1,\dots,T_{n+1}).$

Note that the mapping $\Psi=(\Psi_1,\dots,\Psi_{n+1}): \Bbb K^n \times
\Bbb K\to \Bbb K^{n+1}$ is locally finite outside the set $V(J)\times \Bbb K$. 
Consider $\Bbb K^{n+1}$ as a product $\Bbb K^{n-q}\times \Bbb K^{q+1}$, and let us consider in this product 
coordinates $(y_{q+2},...,y_{n+1}, y_1,...,y_{q+1}).$ Hence $\Psi$ restricted to $V(J)\times \Bbb K$ coincides with the mapping:
$ (x,z)\mapsto (l_1(x),..., l_{n-q}(x), x_1,...,x_{q+1})$ (recall that we consider a new generic system of coordinates).
Let $\phi'=0$ describes the image of the projection 
$$\pi: V(J)\ni x\mapsto (x_{1},..., x_{{q+1}})\in {\Bbb K}^{q+1}.$$ Put $S=\{ T\in \Bbb K^{n+1} : \phi'(T)=0\}$. 
Hence $V(J)\times \Bbb K$ is contained in $\Psi^{-1}(S).$ Consequently the mapping
 $\Psi$ is proper outside the 
hypersurface $S$ and thus the set of non-properness of the mapping $\Psi$  is contained in the  $S.$ 

Since the mapping $\Psi$ is finite
outside $S$, for every  $H\in k [x_1,\dots,x_n, z]$ there is
a minimal polynomial $P_H(T,Y)\in k[T_1,\dots,T_{n+1}][Y]$
such that $P_H(\Psi_1,\dots,\Psi_{n+1},H)= \sum^r_{i=0}
b_i(\Psi_1,\dots,\Psi_{n+1})H^{r-i}=0$ and the coefficient $b_0$
satisfies $\{ T : b_0(T)=0\}\subset S$. In particular $b_0$ depends only on variables $x_1,..., x_{q+1}.$ Moreover, $P_H$ describes a hypersurface given by parametric 
equation $(\Psi_1,...,\Psi_{n+1},H)$ and by Gr\"obner base computation we see that we can assume $P_H(T,Y)\in k[T_1,\dots,T_{n+1}][Y]$.
Now set $H=z.$ 

We
have 
$${\rm deg}_T {P_z}(T_1^{d_1},T_2^{d_2},\dots,T_n^{d_n},
T_{n+1},Y)\le d_s\prod^{n-q-1}_{j=1} d_j$$ and consequently we obtain
the equality $b_0(x_{1},...,x_{q+1})+\sum_{i=1}^{n-q} F_i g_i=0,$ where deg $F_i g_i\le
 d_s\prod^{n-q-1}_{j=1} d_j$. Set $\phi=b_0$. By the construction the
polynomial $\phi$ has zeros only on the image of the projection 
$$\pi: V(J)\ni x\mapsto (x_{1},..., x_{{q+1}})\in {\Bbb K}^{q+1}.$$  $\Box$

\vspace{3mm}

\begin{re}
{\rm Simple application of the Bezout theorem shows that our bound on the degree of $\phi$ is sharp.}
\end{re}

\begin{co}
Let $k$ and $I$ and system of coordinates be as above. If $V(I)$ has pure dimension $q$ and $I$ has not embedded components, then there is a polynomial $\phi_1\in k [x_{1},...,x_{q+1}]$ which describes the image of the projection 
$$\pi: V(I)\ni x\mapsto (x_{1},..., x_{q+1})\in {\Bbb K}^{q+1}$$ such that  

(a) $\phi_1\in I$,

(b) $\deg \phi_1 \le d_s\prod^{n-q-1}_{i=1} d_i.$
\end{co}

\begin{proof}
The set $V(J)=q$ has pure dimension $q$. Consequently $\pi(V(J))$ and $\pi(V(I))$ are hypersurfaces. 
Moreover, by Gr\"obner bases computation the set $\pi(V(I))$ is described by a polynomial $\psi$ from $k[x_1,...,x_{q+1}].$ 
Let $\phi$ be a polynomial as above which vanishes exactly on   $\pi(V(J)).$  Let $\phi_1$ be a product of all irreducible  factors of $\phi$
(over the field $k$) which divides $\psi.$ Hence  $\phi=\phi_1\phi_2, \ \phi_1,\phi_2\in k[x_1,...,x_{q+1}]$, where  $\phi_2$ does not vanish on any component of $V(I).$ Let $I=\bigcap^r I_s$ be a primary decomposition of $I$. Consequently $\phi_1\in I_j$ for every $s$ (by properties of primary ideals) and consequently $\phi_1\in I.$ But $\phi_1$ describes the image of the projection 
$$\pi: V(I)\ni x\mapsto (x_{1},..., x_{q+1})\in {\Bbb K}^{q+1}.$$
\end{proof}

\section{A deformation argument}

In this section, we improve Theorem \ref{skzera} by releasing the necessity of a generic change of coordinates, so
conditions (a) and (b) will be satisfied in the initial  system of coordinates.

\begin{theo}\label{NewMainTh}
Let $k$ be an infinite field and let $f_1,\dots,f_s\in k[x_1,\dots,x_n]$ be polynomials such  $\deg
f_i=d_i$ where $d_1\ge d_2\ge...\ge d_s.$
Assume that $I=(f_1,\dots,f_s)\subset k[x_1,\dots,x_n]$ is a non zero ideal, such that $V(I)$ has  dimension $q\ge 0$.  
There exist polynomials $g_{j}\in k[x_1,\dots,x_n]$  and  a non-zero
polynomial $\phi \in k[x_{1},...,x_{q+1}]$ such that

(a) $ \deg g_{j}f_j\le d_s\prod^{n-q-1}_{i=1} d_i,$

(b) $\phi=\sum^{s}_{j=1} g_{j}f_j$. 

\end{theo}
\begin{proof}
We use  Theorem \ref{skzera}, but over the field  $\Bbb L=k(t)$.
We consider a new generic change of coordinates using generic values $a_{i,j}$ in the infinite field $k$, 
together with the inverse change of coordinates
$$  X_{i}=x_{i} + t \sum_{j=i+1}^n a_{i,j}x_j \; ; \; x_{i}=X_{i} + t \sum_{j=i+1}^n b_{i,j}(t) X_j,
$$ where $b_{i,j}(t)\in k[t].$

As in the proof of Theorem \ref{skzera}, we obtain some
polynomials $G_{j}\in  \Bbb L[X_1,\dots,X_n]$  and  a non-zero
polynomial $b_0 \in  \Bbb L[X_{n-q},...,X_{n}]$ such that, after chasing the denominators, 
$$  b_0(X,t)=\sum^{n-q}_{j=1} G_j(X,t) \bar{F_j}(X,t),$$ where $ b_0(X,t), G_j(X,t), \bar{F_j}(X,t)\in k[t][X_1,...,X_n].$

We cannot just  simplify this equality  by $t$ and then set $t=0$, because  we cannot  exclude the possibility
that there will be a  remaining factor $t^p$ in the left hand side with $p$ strictly positive.
To rule out this possibility, we need  to perform several reduction steps. 
Consider the sub-module $M=\{H(x)=(H_1(x), \dots, H_{n-q}(x)) \}$ of $k[x]^{n-q}$  formed by the relations (first syzygies)
between the polynomials $ F_1(x), \dots, F_{n-q}(x)$. To each element $H(x)$ in $M$ corresponds via the change of coordinates 
a relation $\bar{H}(X,t)$ between the polynomials $\bar{F}_1(X,t), \dots, \bar{F}_{n-q}(X,t)$, such that
$ \bar{H}(X,t)-{H}(X)$ is divisible by $t$.
Re-writing in $(x,t)$, we obtain that
$$  b_0(X,t)=\sum^{n-q}_{j=1} (G_j(X,t)- \bar{H_j}(X,t)) \bar{F_j}(X,t).
$$
We may assume that in the previous equality $b_0(X,t)$ has the form $ b_0(X,t)  = t^p (\phi(x) +t \phi_1(x,t))$
; notice that the $x-$ degree of $\phi(x)$ is bounded by the $X-$degree of $ b_0(X,t))$. 
Each reduction step will produce a similar
equality (with the same degree in $x$ bounds) but with a strictly smaller power $p$.

Assume $p>0$ and let $t=0$, we obtain a non trivial relation $ 0=\sum^{n-q}_{j=1} G_j(x,0) F_j(x)$, hence 
 $H=(G_1(x,0), \dots, G_{n-q}(x,0))$, a non trivial element of $M$. Notice that the $x-$ degree of $ G_j(x,0)$ is bounded by the $X-$degree of $G_j(X,t)$. 
 To which we associate its $\bar{H}$ as above with the same degree bound in $X$ (equivalently in $x$
 by linearity) and notice that
 now $\sum^{n-q}_{j=1} (G_j(X,t)- \bar{H_j}(X,t)) \bar{F_j}(X)$ vanishes for $t=0$, hence admits a factor $t$.
 We can simplify the two sides of the previous equality by $t$ and obtain
 
 $ t^{p-1} (\phi(x) +t \phi_1(x,t))= \sum^{n-q}_{j=1} (G_j(X,t)- \bar{H_j}(X,t)) \bar{F_j}(X,t).$ 

 \noindent After at most $p$ such reduction steps, we get rid of the initial factor $t^p$ and setting $t=0$, we obtain the announced equality with the 
 announced bounds. 
 \end{proof}

    \end{document}